\documentclass[12pt]{amsart}

\usepackage{amssymb,amsmath,graphicx,color,amsthm,bbm,bbold,enumerate}
\usepackage{xcolor}

\newcommand{\ignore}[1]{}

\usepackage{algorithm,algorithmic}

\newtheorem{example}{Example}

\newcommand{\secref}[1]{Section \ref{sec:#1}}

\numberwithin{equation}{section}

\def\Ed{{\mathbb{E}}}

\newcommand{\N}{\mathbb{N}}

\renewcommand{\P}{\mathbb{P}}

\renewcommand{\d}{\delta}
\newcommand{\e}{\varepsilon}

\newcommand{\s}{\sigma}

\renewcommand{\l}{\lambda}

\newcommand{\D}{\Delta}

\newcommand{\Coloc}{\mathcal{C}^{loc}}
\newcommand{\Co}{\mathcal{C}}
\newcommand{\Coglo}{\mathcal{C}^{glo}}

\newcommand{\be}{\begin{equation}}
\newcommand{\ee}{\end{equation}}

\newtheorem{teorema}{Theorem}
\newtheorem{proposicao}[equation]{Proposition}
\newtheorem{definition}{Definition}
\newtheorem{lema}{Lemma}
\newtheorem{corolario}[equation]{Corollary}

\newtheorem{obs}{Remark}
\newenvironment{claim}[1]{\par\noindent\textit{Claim:}\space#1}{}
\newenvironment{claimproof}[1]{\par\noindent\textit{Proof of the claim:}\space#1}{\hfill $\blacksquare$}
\def\sF{\mathcal{F}}
\renewcommand{\Pr}[1]{\mathbb{P}\left(#1\right)}

\begin{document}
\title{Disparity of clustering coefficients in the Holme-Kim network model}
\author{Roberto I. Oliveira, Rodrigo Botelho Ribeiro, Remy Sanchis}

\begin{abstract}The Holme-Kim random graph processes is a variant of the Barab\'{a}si-\'{A}lbert scale-free graph that was designed to exhibit clustering. In this paper we show that whether the model does indeed exhibit clustering depends on how we define the clustering coefficient. In fact, we find that local clustering coefficient remains typically positive whereas global clustering tends to $0$ at a slow rate. These and other results are proven via martingale techniques, such as Freedman's concentration inequality combined with a bootstrapping argument.\end{abstract}
\maketitle

\section{Introduction}

The theory of Random Graphs traces its beginnings to the seminal work of Erd\"{o}s and Renyi  \cite{ER1}. For about 40 years, most papers in this area were about homogeneous models, where vertices are statistically indistinguishable. The original Erd\"{o}s R\'{e}nyi model and random regular graphs both fall in this category \cite{BollobasRG,JLR}. 

More recently, inhomogeneous random graphs have also gained prominence, as large-scale networks of social, biological and technological origins tend to be very far from homogeneous. Research on models of such ``complex networks"~has attracted much interest in Statistical Physics, Systems Biology, Sociology and Computer Science, as well as in Mathematics. We will not attempt to survey this huge area of research, but the interested reader is directed to \cite{Newman} for a broad survey of the non-rigorous literature and to \cite{ChungLu,DurrettRGD,vdH} for compendia of rigorous results for many different models.

Two important complex network features are especially important for this paper. The first one is {\em power-law}, or {\em scale-free}, degree distribution. It is believed that many real-life networks have constant average degree (relative to network size) but a highly skewed degree distribution, where the fraction of nodes of degree $k$ behaves roughly as $k^{-\beta}$ for some~$\beta>0$. This contrasts with the Poisson degree distribution of sparse Erd\"{o}s-R\'{e}nyi graphs. A well known scale-free model is the seminal Barab\'{a}si-\'{A}lbert preferential attachment random graph model \cite{AB,BollobasEtAl_Degree}. In this case $\beta=3$ is the power law exponent, but any $2<\beta\leq 3$ may be achieved by minor modifications of the process \cite{BuckleyOsthus_Degree}. 

Another important feature of certain social and metabolic networks is so-called {\em clustering}, meaning  that ``friends of friends tend to be friends." That is, if a node $v$ is connected to both $w$ and $u$, then $u$ and $w$ are likely to be connected. This trait is not captured by the Barab\'{a}si-\'{A}lbert model, but the Small World model of Strogatz and Watts \cite{SW}, which is also extremely well-known, does have this property (but has no power law).


\subsection{The Holme-Kim model}\label{sec:thehkmodel}

In this paper we provide a rigorous analysis of a specific non-homogeneous random graph model, whose motivation was to combine scale-freeness and clustering. This model was introduced in 2001 by Holme and Kim \cite{HK}. The Holme-Kim or HK model describes a random sequence of graphs $\{G_t\}_{t\in\N}$ that will be formally defined in \secref{HKdef}. Here is an informal description of this evolution. Fix two parameters $p\in [0,1]$ and a positive integer~$m>1$, and start with a graph $G_1$. For $t>1$, the evolution from $G_{t-1}$ to $G_t$ consists of the addition of a new vertex $v_t$ and $m$ new edges between $v_t$ and vertices already in $G_{t-1}$. These $m$ edges are added sequentially in the following fashion. 

\begin{enumerate}
\item The first edge is always attached following the {\em preferential attachment (PA)} mechanism, that is, it connects to a previously existing node $w$ with probability proportional to the degree of $w$ in $G_{t-1}$. 
\item Each of the $m-1$ remaining edges makes a random decision on how to attach.
\begin{enumerate}
\item With probability $p$, the edge is attached according to the {\em triad formation (TF)} mechanism. Let $w'$ be the node of $G_{t-1}$ to which the previous edge was attached. Then the current edge connects to a neighbor $w$ of $w'$ chosen with probability proportional to number of edges between $w$ and $w'$.
\item With probability $1-p$ ($p\in [0,1]$ fixed) the edge follows the same PA mechanism as the first edge (with fresh random choices). 
\end{enumerate}
\end{enumerate}

The case $p=0$ of this process, where only preferential attachment steps are performed, is essentially the Barab\'{a}si-\'{A}lbert model \cite{AB}. The triad formation steps, on the other hand, are reminiscent of the copying model by Kumar et al. \cite{Kumar}. Holme and Kim argued on the basis of simulations and non-rigorous analysis that their model has the properties of scale-freeness and positive clustering. 

Our rigorous results partly confirm their findings. The degree power law can be checked by known methods. On the other hand, we show that there were aspects of the clustering phenomenon (or lack thereof) that were not made evident in \cite{HK} or in other papers in the large networks literature (e.g. \cite{ORS1}). We will see that the question of whether the Holme-Kim model~has clustering admits two different answers depending on how we define clustering. In fact, we will argue that this same phenomenon should hold for a wide variety of network models. 

\subsection{Two distinct clustering coefficients} To clarify this, we will need some graph-theoretic definitions, valid for any graph $G$. We state somewhat informally below and give precise definitions in \secref{prelim}.\begin{itemize}
\item A {\em triangle} in $G$ is a set of three vertices that are mutually connected by edges. 
\item A {\em cherry} (or path of length two) in $G$ is a set of three vertices $\{u,v,w\}$ of $G$ where $u$ and $w$ are both adjacent to $v$ (we call $v$ the middle vertex).  
\item The {\em local clustering coefficient} $\mathcal{C}_{G}^{loc} $ measures the average, over vertices $v$ of $G$, of the fraction of pairs of neighbors of $v$ that are connected by edges. That is, if each pair of neighbors of $v$ counts as a ``potential triangle involving $v$", then $\mathcal{C}_{G}^{loc}$ is average over vertices $v$ of the ratio of ``actual triangles involving $v$"~to ``potential triangles~involving $v$."
\item The {\em global clustering coefficient} $\mathcal{C}_{G}^{glo}$ measures the fraction of all cherries $u\sim v\sim w$ in $G$ that also satisfy $u\sim w$. If we count each cherry $u\sim v\sim w$ as a ``potential triangle" in $G$, then $\mathcal{C}_{G}^{glo}$ measures the ratio of actual to potential triangles in the whole of $G$, multiplied by a normalization constant so that $0\leq \Coglo_{G}\leq 1$ (note that each actual triangle contains several cherries, if vertex labels are taken into account).
\end{itemize}    
Bollob\'{a}s and Riordan \cite{BR} observed that $\Coloc_G$ and $\Coglo_G$ are used interchangeably in the non-rigorous literature. They warned that:
\begin{quotation}{\em 
In more balanced graphs the definitions will give more similar values, but they will still differ by at least a constant factor much of the time \cite{BR}.}
\end{quotation}
In fact, more extreme differences are possible for non-regular graphs.
\begin{example}Build a graph $G$ consisting of an edge $e$ and $n-2$ other vertices connected to the two endpoints of $e$, it is easy to see that $\Coloc_G = 1 - \frac{2}{n}$. On the other hand, it is straightforward to see that $\Coglo_G = \frac{3(n-2)}{n-2+(n-1)(n-2)}=\frac{3}{n}$.
\end{example} 

\subsection{Our results} The main results in this paper show that such a disparity between local and global clustering does indeed occur in the specific case of the Holme-Kim model, albeit in a less extreme form than what the Example suggests. 

\begin{teorema}[Positive local clustering for HK]\label{thm:local}Let $\{G_t\}_{t\geq 0}$ be the sequence of graphs generated by the Holme-Kim model with parameters $m\geq 2$ and $p\in (0,1)$. Then there exists $c>0$ depending only on $m$ and $p$ such that the local clustering coefficients $\Coloc_{G_t}$ of the graphs $G_t$ satisfy:
\[\lim_{t\to +\infty}\Pr{\Coloc_{G_t}\geq c}=1.\]\end{teorema}

\begin{teorema}[Vanishing glocal clustering for HK]\label{thm:global} Let $\{G_t\}_{t\geq 0}$ be as in Theorem \ref{thm:local}. Then there exist constants $b_1,b_2>0$ depending only on $m$ and $p$ such that the global clustering coefficients $\Coglo_{G_t}$ satisfy:
\[\lim_{t\to +\infty}\Pr{\frac{b_1}{\log t}\leq \Coglo_{G_t}\leq \frac{b_2}{\log t}}=1.\]\end{teorema}

Thus for large $t$, one of the two clustering coefficients is typically far from $0$, whereas the other one goes to $0$ in probability, albeit at a small rate. This shows that the remark by Bollob\'{a}s and Riordan is very relevant in the analysis of at least one network model. Our results contradict the numerical findings of \cite{ORS1}, where the Holme-Kim model is cited as a model that exhibits positive ``clustering coefficient'', but the definition of clustering used corresponds to the global coefficient\footnote{Their error may be partly explained by the slow decay of $\Coglo_{G_t}$.}.

For completeness, we will also check in the Appendix that the HK model is scale-free with power-law exponent $\beta=3 $. The proof follows from standard methods in the literature.

\begin{teorema}[The power-law for HK] \label{teo:powerlaw}Let $\{G_t\}_{t\geq 0}$ be as in the previous theorem. Also let $N_t(d)$ be the number of vertices of degree $d$ in $G_t$ and set 
\[D_t(d):=\frac{\mathbb{E}{N_{t}(d)}}{t}.\]
Then
\[ 
\lim_{t \rightarrow \infty}D_t(d) = \frac{2(m+3)(m+1)}{(d+3)(d+2)(d+1)}.
\] and \[\mathbb{P}\left( \left| N_t(d) - D_t(d)\,t \right| \ge 16d\cdot c\cdot \sqrt{t} \right) \le (t+1)^{d-m}e^{-c^2}.
\]
\end{teorema}

\subsection{Heuristics and a seemingly general phenomenon}\label{sec:heuristics} The disparity between $\Coloc_G$ and $\Coglo_G$ should be a {\em general phenomenon for large scale-free graph models with many (but not too many triangles).} This will transpire from the follwowing heuristic analysis of the Holme-Kim case with $p\in (0,1)$. 

To begin, it is not hard to understand why Theorem \ref{thm:local} should be true. By Theorem \ref{teo:powerlaw}, there is a positive fraction of nodes with degree $m$. Moreover, a positive fraction of these vertices are contained in at least one triangle because of TF steps. A more general observation may be made.

\begin{quotation}{\em  {\bf Reason for positive local clustering:} if a positive fraction of nodes have degree $\leq d$ (assumed constant), and a positive fraction of these nodes are contained in at least one triangle, then the local clustering coefficient $\Coloc_{G_t}$ must be bounded away from zero.} \end{quotation} 

We now argue that {\em the vanishing of $\Coglo_{G_t}$ should be a consequence of the power law degree distribution}. The global clustering coefficient $\Coglo_{G_t}$ is essentially the ratio of the number of triangles to the number of cherries in $G_t$, the latter being denoted by $C_t$. Now, one can easily show that the number of triangles in $G_t$ grows linearly in $t$ with high probability, so: 
\[\Coglo_{G_t} \approx \frac{\mbox{\# of triangles in $G_t$}}{C_t}\approx \frac{t}{C_t}.\]

To estimate $C_t$, we note that each vertex $v$ of degree $d$ in $G_t$ is the ``middle vertex"~of exactly $d(d-1)/2\approx d^2$ cherries. This means
\begin{eqnarray*}\frac{C_t}{t} &\approx & \sum_{d=1}^t\,\frac{N_t(d)}{t}\,d^2\\ 
\left(\mbox{$\frac{N_t(d)}{t}\sim D_t(d) \approx d^{-3}$ by Theorem \ref{teo:powerlaw}}\right) &\approx & \sum_{d=1}^t \frac{1}{d} \approx \log t.\end{eqnarray*}
Our reasoning is not rigorous because it requires bounds on $N_t(d)$ for very large $d$. However, we feel our argument is compelling enough to be true for many models. In fact, considering the case where $N_t(d)/t\approx d^{-\beta}$ for $0<\beta\leq 3$, one is led to the following.

\begin{quotation}{\em {\bf Heuristic reason for large number of cherries:} if the fraction of nodes of degree $d$ in $G_t$ is $\approx d^{-\beta}$ for some $0< \beta\leq 3$, the number of cherries $C_t$ is superlinear in $t$. More precisely, we expect $C_t/t \approx t^{3-\beta}$ for $0<\beta<3$ and $C_t/t\approx \log t $ for $\beta=3$.}\end{quotation}

The power law range $0<\beta\leq 3$ corresponds to most models of large networks in the literature. Likewise, we believe that the disparity between $\Coglo_{G_t}$ and $\Coloc_{G_t}$ should hold for all ``natural"~random graph sequences with many triangles and power law degree distribution with exponent $0<\beta\leq 3$.  The general message is this.

\begin{quotation}{\em {\bf Heuristic disparity between local and global clustering:} achieving positive local clustering is ``easy": just introduce a density of triangles in sparse areas of the graph. On the other hand, if the number of triangles in $G_t$ grows linearly with time, and the fraction of nodes of degree $d$ in $G_t$ is $\approx d^{-\beta}$ for some $0< \beta\leq 3$, then one expects a vanishingly small global clustering coefficient.}\end{quotation}

\subsection{Main technical ideas} At a high level, our proofs follow standard ideas from previous rigorous papers on complex networks. For instance, suppose one wants to keep track of the number of nodes of degree $d$ at time $t$, for $d=m,m+1,\dots,D$. Letting $N_t=(N_t(m),\dots,N_t(D))$, the basic strategy adopted in several previous papers is to find a deterministic matrix $\mathcal{M}_{t-1}$ and a deterministic vector $r_{t-1}$, both measurable with respect to $G_0,\dots,G_{t-1}$, such that:
\[N_t = \mathcal{M}_{t-1}\,N_t + r_{t-1} + \epsilon_t, \mbox{ where }\mathbb{E}(\epsilon_t\mid G_0,\dots,G_t)\approx 0.\]
This can be seen as ``noisy version"~of the deterministic recursion $N_t = \mathcal{M}_{t-1}N_{t-1} + r_{t-1}$ with $\epsilon_t$ the ``noise"~term. One then studies the recursion and uses martingale techniques (especially the Azuma-H\"{o}ffding inequality) to prove that $N_t$ concentrates around the solution of the deterministic recursion. Our own proof of the degree power law follows this outline, and is only slightly different from the one in \cite{ChungLu}. 

Once the degree sequence is analyzed, Theorem \ref{thm:local} is then a matter of observing that a density of vertices of degree $m$ will be contained in at least one triangle, due to a TF step. On the other hand, the analysis of global clustering is harder due to the need to estimate the number of cherries $C_t$. Justifying the heuristic calculation above would require strong control of the degree distribution up to very large values of $d$. We opt instead to write a ``noisy recursion"~for $C_t$ itself. However, the increments in this noisy recursion can be quite large, and the Azuma-H\"{o}ffding inequality is not enough to control the process. We use instead Freedman's concentration inequality, which involves the quadratic variation, but even that is delicate because the variation might ``blow up"~in certain unlikely events. In the end, we use a kind of ``bootstrap" argument, whereby a preliminary estimate of $C_t$ is fed back into the martingale calculation to give sharper control of the predictable terms and the variation. The upshot is a Weak Law of Large Numbers for $C_t$, given in Theorem \ref{teo:ctweaklaw} below:
\begin{teorema}[The Weak Law of large Numbers for $C_t$]\label{teo:ctweaklaw}Let $C_t$ be the number of cherries in $G_t$, then
\[
\frac{C_t}{t\log t} \stackrel{\mathbb{P}}{\longrightarrow} \binom{m+1}{2}.
\]
\end{teorema}
Overall, we believe our martingale analysis of $C_t$ is our main technical contribution.
\subsection{Organization.} 
The remainder of this paper is organized as follows. Section \ref{sec:prelim} reviews some standard notation, introduces the relevant graph theoretic concepts and records martingale inequalities. Section~\ref{sec:HKdef} presents a formal definition of the model. In Section~\ref{sec:degrees} we prove technical estimates for the degree which will be useful throughout the paper. Sections \ref{sec:clusloc} and~\ref{sec:cluglo} are devoted to prove the bounds for the local and the global clustering coefficients, respectively. Section \ref{sec:finalcomm} presents a comparative explanation for the so distinct behavior of the clustering coefficients. The proof of the power-law distribution is left to the appendix because it follows well known martingale arguments.

\section{Preliminaries}\label{sec:prelim}

\subsection{Set and multiset notation}\label{sec:notation} $\N=\{0,1,2,3,\dots\}$. For $n\in\N\backslash\{0\}$, $[n]:=\{1,2,\dots,n\}$.
Given a set $S$, $|S|$ denotes its cardinality, and $\binom{S}{k}$ denotes the collection of all subsets of $S$ of size $k\in\N$. The binomial $\binom{n}{k}$ is the number of elements in $\binom{[n]}{k}$. 

A multiset $M$ consists of a base set $M_0$ and, for each $s\in M_0$, a {\em multiplicity} $m_s\in\N\backslash\{0\}$. We say that a multiset $M$ is contained in set $S$ (and write $M\subset S$) if $M_0\subset S$, and we say $s\in M$ if $s\in M_0$.

\subsection{Basic graph theory}
Recall that a graph $G=(V_G,E_G)$ consists of a set $V_G$ of vertices and a multiset $E_G\subset \binom{V_G}{2}$ of edges. This implies that there might be multiple ``parallel"~edges between any given $v,w\in V_G$.  Given $G$ and $v,w\in V_G$, we say $v$ and $w$ are {\em neighbors}, and write $v\sim_G w$, if $\{v,w\}\in E_G$. We also write $\Gamma_G(v) = \{w\in V\,:\, w\sim_G v\}$ for the neighborhood of $v\in V_G$ and $e(\Gamma_G(v))$ for the number of edges between the neighbors of~$v$. Since we allow multiple edges, for all vertices $v,w \in V_G$, we let $e_{G}(v,w)$ be the number of edges between $v$ and $w$. In this case, we may define the \textit{degree} of $v$ in terms of $e_{G}(v,w)$ writing it as $d_G(w)=\sum_{w \in \Gamma_G(v)} e_{G}(v,w)$, i.e., the total of edges whose one of its ends is $v$.

\subsection{Triangles and cherries} \label{sec:trianglesandcherries}

A triangle in a graph $G=(V_G,E_G)$ is a subset $\{u,v,w\}\in \binom{V_G}{3}$ with $u\sim_G v\sim_G w\sim_G u$. We denote the number of triangles contained in a graph $G$ by $\Delta_G$. For a fixed vertex $v \in V_G$, we denote the number of triangles sharing at least the common vertex $v$ by $\Delta_G(v)$.

 A {\em cherry}, or path of length two, is an element $(v,\{u,w\})\in V_G\times \binom{V_G}{2}$ with $u\sim_G v\sim_G w$, or a pair $(v,u)\in V_G^2$ with $e_G(u,v)\geq 2$.  We let $C_G$ denote the number of cherries in $G$, counted according to edge multiplicities:
\[C_G:= \sum_{v\in V_G}\,\binom{d_G(v)}{2}.\]

\subsection{Clustering coefficients}\label{sec:clustering} We assume $G=(V_G,E_G)$ is a graph where $|V_G|>0$ and all vertices have degree at least $2$. 
\begin{definition}[Local clustering coefficient in $v$]
Given a vertex $v \in G$, the local clustering coefficient at $v$, is
\[ \Co_G(v) = \frac{ \Delta_G(v)}{\binom{d_{G}(v)}{2}}.\]\end{definition}
Notice that $0\leq \Co_G(v)\leq 1$ always, since there can be at most one triangle formed by $v$ and a pair of its neighbors. In probabilistic terms, $\Co_G(v)$ measures the probability that a pair of random neighbors of $v$ form an edge, ie. how likely it is that ``two friends of $v$ are also each other's friends."

The two coefficients for the graph $G$ are as follows.
\begin{definition}[Local and Global Clustering Coefficients]\label{LC}The local clustering coefficient of $G$ is defined as
  \[ \Coloc_G:= \frac{\sum_{v\in V_G}\,\Co_G(v)}{|V_G|},\]
whereas the global coefficient is 
  \[ \Coglo_G:= 3\times\frac{\Delta_G}{C_G}.
  \]
\end{definition}


\begin{obs} For our purposes there is no significant difference between taking or not multiple edges into account in the above definitions. The key point is that our model allows at most $m$ edges between two vertices. Thus all bounds considering multiple edges may be carried over to the other case (and vice versa) at the cost of losing constant factors.

\end{obs}

\subsection{Sequences of graphs} We will often consider the sequence of graphs $\{G_t\}_{t}$ defined by the Holme-Kim model, which is indexed by a discrete time parameter $t\geq 0$. When considering this sequence, we replace the subscript $G_t$ with $t$ in our graph notation, so that $V_t:=V_{G_t}$, $E_t:=E_{G_t}$, etc. Given a sequence of numerical values $\{x_t\}_{t\geq 0}$ depending on $t$, we will let $\Delta\,x_t:=x_{t}-x_{t-1}$.

\subsection{Asymptotics} We will use the Landau $o/O/\Theta$ notation at several points of our discussion. This always pressuposes some asymptotics as a parameter $n$ or $t$ of interest goes to $+\infty$. Just which parameter this is will always be clear from context.

\subsection{Martingale concentration inequalities}

Recall that a supermartingale $(M_n,\sF_n)_{n\in\N}$ consists of a filtration $\sF_0\subset\sF_1\subset\sF_3\subset \dots $ and a family $\{M_n\}_{n\in\N}$ of integrable random variables where, for each $n\in\N$, $M_n$ is $\sF_n$-measurable and $\mathbb{E}[M_{n+1}\mid \sF_n]\leq M_n$. If $(-M_n,\sF_n)_{n\in\N}$ is also a supermartingale, we say $(M_n,\sF_n)_{n\in\N}$ is a martingale. We recall two well-known concentration inequalities for martingales. 

\begin{teorema}[Azuma-H\"{o}ffding Inequality - \cite{ChungLu}] Let $(M_n,\sF_n)_{n \ge 1}$ be a (super)martingale satisfying
\[ 
\lvert M_{i+1} - M_i \rvert \le a_i
\]
Then, for all $\l > 0 $ we have
\[
\P \left( M_n - M_0 > \l \right) \le \exp\left( -\frac{\l^2}{\sum_{i=1}^n a_i^2} \right).
\]
\end{teorema}

\begin{teorema}[Freedman's Inequality - \cite{freedman1975}]\label{teo:azumaquadratic} Let $(M_n, \mathcal{F}_n)_{n \ge 1}$ be a (super)martingale. Write 
\[
V_n := \sum_{k=1}^{n-1} \mathbb{E} \left[(M_{k+1}-M_k)^2\middle|\mathcal{F}_k \right]
\]
and suppose that $M_0 = 0$ and 
\[ 
\lvert M_{k+1} - M_k \rvert \le R, \textit{ for all }k.
\]
Then, for all $\l > 0 $ we have
\[
\P \left(   M_n \ge \l, V_n \le \s^2, \textit{ for some }n\right) \le \exp\left( -\frac{\l^2}{2\s^2 + 2R\l /3} \right). 
\]
\end{teorema}

\section{Formal definition of the process}\label{sec:HKdef}

In this section we give a more formal definition of the Holme-Kim process (compare with Section \ref{sec:thehkmodel}). 

The model has two parameters: a positive integer number $m \ge 2$ and a real number $p\in [0,1]$. It produces a graph sequence $\{ G_t \}_{t\ge 1}$ which is obtained inductively according to the growth rule we describe below. 

\par{\sc Initial state.} The initial graph $G_1$, which will be taken as the graph with vertex set $V_1=\{1\}$ and  a single edge, which is a self-loop. 

\par{\sc Evolution.} For $t>1$, obtain $G_{t+1}$ from $G_t$ adding to it a vertex ${t+1}$ and $m$ edges between $t+1$ and vertices $Y^{(i)}_{t+1}\in V_t$, $1\leq i\leq m$. These vertices are chosen as follows. Let $\sF_t$ be the $\sigma$-field generated by all random choices made in our construction up to time $t$. Assume we are given i.i.d. random variables  $(\xi_{t+1}^{(i)})$ independent from $\sF_t$. We define:
\[
\mathbb{P}\left( Y^{(1)}_{t+1} = u \middle | \sF_t \right) = \frac{d_t(u)}{2mt},
\]
which means the first choice of vertex is always made using the \textit{preferential attachment} mechanism. The next $m-1$ choices $Y^{(i)}_{t+1}$, $2\leq i\leq m$, are made as follows: let $\sF_t^{(i-1)}$ be the $\sigma$-field generated by $\sF_t$ and all subsequent random choices made in choosing $Y^{(j)}_{t+1}$ for $1\leq j\leq i-1$. Then:
\[
\mathbb{P}\left( Y^{(i)}_{t+1} = u \middle |\xi_{t+1}^{(i)}=x,\sF^{(i-1)}_t \right) = \left\{ 
\begin{array}{ll}
		\frac{d_t(u)}{2mt}  & \mbox{if } x = 0,  \\ \\
		\frac{e_t(Y^{(i-1)}_{t+1},u)}{d_t(Y^{(i-1)}_{t+1})} & \mbox{if } x = 1 \mbox{ and } u\in \Gamma_t(Y^{(i-1)}_{t+1}), \\ \\
		0 & \mbox{otherwise. }
	\end{array}
\right.
\]
In words, for each choice for the $m-1$ end points, we flip an independent coin of parameter $p$ and decide according to it which mechanism we use to choose the end point. With probability $p$ use the \textit{triad formation} mechanism, i.e., we choose the end point among the neighbors of the previously chosen vertex $Y^{(i-1)}_{t+1}$. With probability $1-p$, we make a fresh choice from $V_t$ using the \textit{preferential attachment} mechanism. In this sense, if $\xi_t^{(i)} = 1$ we say we have taken a \textit{TF-step}. Otherwise, we say a \textit{PA-step} was performed.

\section{Technical estimates for vertex degrees}\label{sec:degrees}

In this section we collect several results on vertex degrees. Subsection \ref{sec:increments} describes the probability of degree increments in a single step. In Subsection \ref{sec:maximumdegrees} we obtain upper bounds on all degrees. Some of these results are fairly technical and may be skipped in a first reading.

\subsection{Degree increments}\label{sec:increments} We begin the following simple lemma.

\begin{lema}\label{lema:probdeltadt} For all $k \in \{ 1,..,m\}$, there exists positive constants $c_{m,p,k}$ such that
\[
\P\left(\Delta d_t(v) = k \middle | G_t \right) = c_{m,p,k}\frac{d_t^k(v)}{t} + O\left(\frac{d^{k+1}_t(v)}{t^{k+1}}\right).
\]
In particular, for $k=1$ we have $c_{m,p,1} = 1/2$.
\end{lema}
\begin{proof} We begin proving the following claim involving the random variables $Y^{(i)}_t$
\begin{claim} For all $i \in \{ 0, 1, 2,\ldots, m \}$, 
\begin{equation}\label{eq:choicestep}
\P\left( Y^{(i)}_{t+1} = v \middle| G_t \right) = \frac{d_t(v)}{2mt}.
\end{equation}
\end{claim}
\begin{claimproof} The proof follows by induction on $i$. For $i=1$ we have nothing to do. So, suppose the claim holds for all choices before $i-1$. Then,
\begin{equation}\label{eq:yit}
\begin{split}
\P\left( Y^{(i)}_{t+1} = v \middle| G_t \right) & = \P\left( Y^{(i)}_{t+1}= v, Y^{(i-1)}_{t+1} = v  \middle| G_t \right) + \P\left( Y^{(i)}_{t+1} = v, Y^{(i-1)}_{t+1} \neq v  \middle| G_t \right)\\
& = \frac{(1-p)}{4m^2}\frac{d^2_t(v)}{t^2} + \P\left( Y^{(i)}_{t+1}= v, Y^{(i-1)}_{t+1} \neq v  \middle| G_t \right).
\end{split}
\end{equation}
For the first term on the \textit{r.h.s} the only way we can choose $v$ again is following a \textit{PA-step} and then choosing $v$ according to preferential attachment rule. This means
\begin{equation}\label{eq:choice}
 \P\left( Y^{(i)}_{t+1}= v, Y^{(i-1)}_{t+1} = v  \middle| G_t \right) = \frac{(1-p)}{4m^2}\frac{d^2_t(v)}{t^2}.
\end{equation}
For the second term, we divide it in two sets, whether the vertex chosen at the previous choice is a neighbor of $v$ or not.   
\begin{equation*}
\begin{split}
\P\left( Y^{(i)}_{t+1}= v, Y^{(i-1)}_{t+1} \neq v  \middle| G_t \right) & = \sum_{u \notin \Gamma_{G_t}(v)}\P\left( Y^{(i)}_{t+1} = v, Y^{(i-1)}_{t+1}=u  \middle| G_t \right)\\
& \quad + \sum_{u \in \Gamma_{G_t}(v)}\P\left( Y^{(i)}_{t+1}= v, Y^{(i-1)}_{t+1}=u  \middle| G_t \right) \\
& = (1-p)\frac{d_t(v)}{2mt}\left( \sum_{u \notin \Gamma_{G_t}(v)} \frac{d_t(u)}{2mt}\right) \\
& \quad + p\left( \sum_{u \in \Gamma_{G_t}(v)} \frac{e_{G_t}(u,v)}{d_t(u)}\frac{d_t(u)}{2mt}\right) \\
& \quad + \frac{(1-p)}{2m}\frac{d_t(v)}{t}\left( \sum_{u \in \Gamma_{G_t}(v)} \frac{d_t(u)}{2mt}\right) \\
& = \frac{d_t(v)}{2mt} - \frac{(1-p)}{4m^2}\frac{d^2_t(v)}{t^2}.
\end{split}
\end{equation*}
We used our inductive hypothesis and that
\[
\sum_{u \in \Gamma_{G_t}(v)} \frac{d_t(u)}{2mt} + \sum_{u \notin \Gamma_{G_t}(v)} \frac{d_t(u)}{2mt} = 1- \frac{d_t(v)}{2mt}
\]
Returning to (\ref{eq:yit}) we prove the claim.
\end{claimproof}

We will show the particular case $k=1$ since we have particular interest in the value of $c_{m,p,1}$ and point out how to obtain the other cases. We begin noticing that the process of choices is by definition a homogeneous Markovian process. This means to evaluate the probability of a vertex increasing its degree by exactly one, the case $k=1$, we just need to know the probabilities of transition. In this way, use the notation $\mathbb{P}_t$ to denote the measure conditioned on $G_t$ and go to the computation of the probabilities of transition. We start by the hardest one.
\begin{equation}\label{eq:npras}
\begin{split}
\mathbb{P}_t\left( Y_{t+1}^{(i+1)}=v \middle | Y_{t+1}^{(i)}\neq v \right) & =  \sum_{u \in \Gamma_{G_t}(v)}\mathbb{P}_t\left( Y_{t+1}^{(i+1)}=v \middle | Y_{t+1}^{(i)} = u \right)\mathbb{P}_t\left( Y_{t+1}^{(i)} = u\middle | Y_{t+1}^{(i)} \neq v \right)\\
&  \quad + \sum_{u \notin \Gamma_{G_t}(v)}\mathbb{P}_t\left( Y_{t+1}^{(i+1)}=v \middle | Y_{t+1}^{(i)} = u \right)\mathbb{P}_t\left( Y_{t+1}^{(i)} = u\middle | Y_{t+1}^{(i)} \neq v \right)\\
\end{split}
\end{equation}
When $u \in \Gamma_{G_t}(v)$, we can choose $v$ taking any of the two steps. This implies the equation below
\begin{equation}
\begin{split}
\mathbb{P}_t\left( Y_{t+1}^{(i+1)}=v \middle | Y_{t+1}^{(i)} = u \right) = (1-p)\frac{d_t(v)}{2mt}+ p\frac{e_t(u,v)}{d_t(u)},
\end{split}
\end{equation}
but when $u \notin \Gamma_{G_t}(v)$ the only way we can choose $v$ is following a \textit{PA-step}, which implies that
\begin{equation}
\begin{split}
\mathbb{P}_t\left( Y_{t+1}^{(i+1)}=v \middle | Y_{t+1}^{(i)} = u \right) = (1-p)\frac{d_t(v)}{2mt}.
\end{split}
\end{equation}
We also notice that the equation below holds, since $u \neq v$ and our claim is true
\begin{equation}
\begin{split}
\mathbb{P}_t\left( Y_{t+1}^{(i)} = u\middle | Y_{t+1}^{(i)} \neq v \right) &= \frac{\mathbb{P}_t\left( Y_{t+1}^{(i)} = u \right)}{\mathbb{P}_t\left( Y_{t+1}^{(i)} \neq v \right)} \stackrel{\textit{Claim}}{=} \frac{d_t(u)}{2mt\mathbb{P}_t\left( Y_{t+1}^{(i)} \neq v \right)}.
\end{split}
\end{equation}
And the same \textit{Claim} also implies that
\begin{equation}
\begin{split}
\frac{1}{\mathbb{P}_t\left( Y_{t+1}^{(i)} \neq v \right)} = \frac{1}{1-\frac{d_t(v)}{2mt}} = 1 + \sum_{n=1}^{\infty}\left(\frac{d_t(v)}{2mt} \right)^n.
\end{split}
\end{equation}
Combining the last three equations to (\ref{eq:npras}), we are able to get
\begin{equation}\label{eq:pns}
\begin{split}
\mathbb{P}_t\left( Y_{t+1}^{(i+1)}=v \middle | Y_{t+1}^{(i)}\neq v \right) & = (1-p)\frac{d_t(v)}{2mt} + p\frac{d_t(v)}{2mt}\left( 1 + \sum_{n=1}^{\infty}\left(\frac{d_t(v)}{2mt} \right)^n \right) \\
& = \frac{d_t(v)}{2mt} + O\left( \frac{d^2_t(v)}{t^2} \right).
\end{split}
\end{equation}
If we chose $v$ at the previous choice, the only way we select it again is following a \textit{PA-step}, this means that
\begin{equation}\label{eq:pss}
\begin{split}
\mathbb{P}_t\left( Y_{t+1}^{(i+1)}=v \middle | Y_{t+1}^{(i)} = v \right) & = (1-p)\frac{d_t(v)}{2mt}.
\end{split}
\end{equation}
From these two probabilities of transition we may obtain the other ones. 

To compute the probability of $\{\Delta d_t(v) =1\}$ given $G_t$ we may split it in $m$ possible ways to increase $d_t(v)$ by exactly one. Each possible way has an index $i \in \{1,...,m\}$ meaning the step we chose $v$ and then we must avoid it at the other $m-1$ choices. This means that each of these $m$ ways has a probability similar to the expression below
\[
\left( 1- \frac{d_t(v)}{2mt} - O\left( \frac{d^2_t(v)}{t^2} \right)\right)^{i-1}\left( \frac{d_t(v)}{2mt} + O\left( \frac{d^2_t(v)}{t^2} \right) \right)\left( 1- \frac{d_t(v)}{2mt} - O\left( \frac{d^2_t(v)}{t^2} \right)\right)^{m-i}
\]
which implies that
\[
\P\left(\Delta d_t(v) = 1 \middle | G_t \right) = \frac{d_t(v)}{2t} + O\left(\frac{d^{2}_t(v)}{t^{2}}\right).
\]

The cases $k> 1$ are obtained in the same way, considering the $\binom{m}{k}$ ways of increase $d_t(v)$ by $k$.
\end{proof}
\begin{obs}\label{obs:bigonotation} We must notice that term $O\left(\frac{d^{k+1}_t(v)}{t^{k+1}}\right)$ given by the Lemma actually is a statement stronger than the real big-O notation. Since the Lemma's proof implies the existence of a positive constant $\tilde{c}_{m,p,k}$ such that 
\[
O\left(\frac{d^{k+1}_t(v)}{t^{k+1}}\right) \le \tilde{c}_{m,p,k}\frac{d^{k+1}_t(v)}{t^{k+1}},
\]
for all vertex $v$ and time $t$.
\end{obs}
\subsection{Upper bounds on vertex degrees}\label{sec:maximumdegrees}
To control the number of cherries, $C_t$, we will need upper bounds on vertex degrees. The bound is obtained applying the Azuma-Hoffding inequality to the degree of each vertex, which is a martingale after normalizing by the quantity defined below.
\begin{equation}\label{ineq:phi}
\phi(t) := \prod_{s=1}^{t-1} \left( 1 + \frac{1}{2s} \right).
\end{equation}
A fact about $\phi(t)$ will be useful: there exists positive constants $b_1$ and $b_2$ such that
\[
b_1\sqrt{t}\le \phi(t) \le b_2\sqrt{t},
\]
for all $t$.
\begin{proposicao}\label{prop:martingaldt} For each vertex $j$, the sequence $(X_t^{(j)})_{t\ge j}$ defined as
\[
X_t^{(j)} := \frac{d_t(j)}{\phi(t)}
\]
is a martingale.
\end{proposicao}
\begin{proof} Since the vertex $j$ will remain fixed throughout the proof, we will write simply $X_t$ instead of $X_t^{(j)}$.

Observe that we can write $d_{t+1}(j)$ as follows
\[
d_{t+1}(j) = d_t(j) + \sum_{k=1}^m \mathbb{1} \left\lbrace Y_{t+1}^{(k)} = j \right \rbrace.
\]
In addition, we proved in Lemma \ref{lema:probdeltadt} that, for all $k \in {1,...,m}$, we have
\begin{equation}\label{eq:yk}
\mathbb{P} \left(Y_{t+1}^{(k)} = j \middle | G_t \right) = \frac{d_t(j)}{2mt}.
\end{equation}
Thus, the follow equivalence relation is true
\begin{equation}\label{eq:expecdt}
\mathbb{E}\left[ d_{t+1}(j) \middle | G_t \right] = \left( 1+ \frac{1}{2t}\right)d_t(j).
\end{equation}
Then, dividing the above equation by $\phi(t+1)$ the desired result follows.
\end{proof}
Once we have Proposition \ref{prop:martingaldt} we are able to obtain an upper bound for $d_t(j)$.
\begin{teorema} \label{teo:upperbounddeg} There is a positive constant $b_3$ such that, for all vertex $j$
\[
\mathbb{P}\left( d_t(j) \ge b_3\sqrt{t}\log(t) \right) \le t^{-100}.
\]
\end{teorema}
\begin{proof}The proof is essentially applying Azuma's inequality to the martingale we obtained in Proposition \ref{prop:martingaldt}. Again we will write it as $X_t$.

Applying Azuma's inequality demands controlling $X_t$'s variation, which satisfies the upper bound below
\begin{equation}\label{ineq:varxs}
	\begin{split}
		\left| \Delta X_s \right| & = \left | \frac{d_{s+1}(j) - \left( 1+ \frac{1}{2s}\right)d_s(j)}{\phi(s+1)}\right | \le \frac{2m}{\phi(s+1)} \le \frac{b_4}{\sqrt{s}}.
	\end{split}
\end{equation}
Thus,
\[
\sum_{s=j+1}^t \left| \Delta X_s \right|^2 \le b_5 \log(t).
\]
We must notice that none of the above constants depend on $j$. Then, Azuma's inequality gives us that
\begin{equation}\label{ineq:probazumadt}
\mathbb{P}\left( \left| X_t - X_0 \right| > \lambda \right) \le 2\exp\left \lbrace -\frac{\lambda^2}{b_5\log(t)}\right \rbrace.
\end{equation}
Choosing $\lambda = 10\sqrt{b_5}\log(t)$ and recalling $X_t = d_t(j)/\phi(t)$ we obtain
\[
\mathbb{P}\left( \left| d_t(j) - \frac{m\phi(t)}{\phi(j)} \right| > 10\sqrt{b_5}\phi(t)\log(t) \right)\le t^{-100}.
\]
Finally, using that $b_1 \sqrt{t} \le \phi(t+1) \le b_2 \sqrt{t}$, comes
\[
\mathbb{P}\left( \left| d_t(j) - \frac{m\phi(t)}{\phi(j)} \right| > 10\sqrt{b_5}b_2\sqrt{t}\log(t) \right) \le t^{-100}.
\]
Implying the desired result.
\end{proof}
An immediate consequence of Theorem ~\ref{teo:upperbounddeg} is an upper bound for the maximum degree of $G_t$

\begin{corolario} [Upper Bound to the maximum degree] \label{cor:dmax}There exists a positive constant $b_1$ such that
\[
\P \left( d_{max}(G_t) \ge b_1 \sqrt{t}\log(t) \right) \le t^{-99}.
\]
\end{corolario}
\begin{proof} The event involving $d_{max}(G_t)$ may be seen as follows
\[
\left \lbrace  d_{max}(G_t) \ge b_1 \sqrt{t}\log(t) \right \rbrace = \bigcup _{j \le t}\left \lbrace d_t(j) \ge b_1\sqrt{t}\log(t) \right \rbrace.
\]
Using union bound and applying Theorem \ref{teo:upperbounddeg} we prove the Corollary.
\end{proof}
The next three lemmas are of a technical nature. Their statements will become clearer in the proof of the upper bound for $C_t$.
\begin{lema}\label{lema:reldtdto} There are positive constants $b_6$ and $b_7$, such that, for all vertex $j$ and all time $t_0 \le t$, we have
\[
\mathbb{P} \left( d_{t_0}(j) > b_4\sqrt{\frac{t_0}{t}}d_t(j) + b_5 \sqrt{t_0}\log(t)\right) \le t^{-100}.
\]
\end{lema}
\begin{proof} For each vertex $j$ and $t_0 \le t$, consider the sequence of random variables $\left(Z_s \right)_{s \ge 0}$ defined as $Z_s = X_{t_0 + s}$, which is adaptable to the filtration $\mathcal{F}_s := G_{t_0 + t}$.

Concerning $Z_s$'s variation, using (\ref{ineq:varxs}), we have the following upper bound
\[
	\left| \Delta Z_s \right| = \left| \Delta X_{t_0+s} \right| \le \frac{b_4}{\sqrt{t_0+s}}.
\]
Thus,
\[
\sum_{s=0}^{t-t_0} \left| \Delta Z_s \right|^2 \le b_5\log(t).
\]
Applying Azuma's inequality, we obtain
\begin{equation}\label{ineq:zt}
	\begin{split}
		\mathbb{P}\left( \left| Z_{t-t_0} - Z_0 \right| \ge \lambda \right) \le 2\exp\left\lbrace - \frac{\l^2}{b_5 \log(t)} \right\rbrace.
	\end{split}
\end{equation}
But, the definition of $Z_s$ and the fact that $\phi(t) = \Theta\left(\sqrt{t}\right)$ the inclusion of events below is true
\[
\left\lbrace d_{t_0}(j) > b_2\sqrt{t_0}\l + \frac{b_2\sqrt{t_0}}{b_1\sqrt{t}}d_t(j) \right \rbrace \subset \left \lbrace \left| Z_t - Z_0 \right| \ge \lambda \right \rbrace,
\]
which, combined to (\ref{ineq:zt}), proves the lemma if we choose $\l = 10\sqrt{b_5}\log(t)$.
\end{proof}
\begin{lema}\label{lema:cotaunidto}
There is a positive constant $b_8$ such that
\[
\mathbb{P}\left(\bigcup_{j = 1}^t \bigcup_{t_0 = j}^t \left\lbrace d_{t_0}(j) > b_8\sqrt{t_0}\log(t) \right\rbrace \right) \le 2t^{-98}.
\]
\end{lema}
\begin{proof}
This Lemma is consequence of Theorem \ref{teo:upperbounddeg} and Lemma \ref{lema:reldtdto}, which state, respectively
\begin{equation}
\mathbb{P}\left( d_t(j) \ge b_3\sqrt{t}\log(t) \right) \le t^{-100},
\end{equation}
\begin{equation}
\mathbb{P} \left( d_{t_0}(j) > b_4\sqrt{\frac{t_0}{t}}d_t(j) + b_5 \sqrt{t_0}\log(t)\right) \le t^{-100}.
\end{equation}
In which the constants $b_3, b_4$ and $b_5$ don't depend on the vertex $j$ neither the times $t_0$ and $t$.

Now, for each $t_0 \le t$ and vertex $j$, consider the events below
\[
A_{t_0,j} := \left\lbrace d_{t_0}(j) > b_8\sqrt{t_0}\log(t) \right\rbrace,
\]
\[
B_{t_0,j} := \left\lbrace d_{t_0}(j) > b_4\sqrt{\frac{t_0}{t}}d_t(j) + b_5 \sqrt{t_0}\log(t) \right\rbrace,
\]
e
\[
C_{t,j} := \left\lbrace d_t(j) \ge b_3\sqrt{t}\log(t)\right\rbrace.
\]
Now we obtain an upper bound for $\mathbb{P} \left( A_{t_0,j}\right)$ using the bounds we have obtained for the probabilities of $B_{t_0,j}$ and $C_{t,j}$. 
\begin{equation}
	\begin{split}
		\mathbb{P} \left( A_{t_0,j}\right) & = \mathbb{P} \left( A_{t_0,j} \cap B_{t_0,j} \right) + \mathbb{P} \left( A_{t_0,j} \cap B^c_{t_0,j} \right) \\
& \le \mathbb{P} \left( B_{t_0,j}\right)+ \mathbb{P} \left( A_{t_0,j} \cap B^c_{t_0,j} \cap C_{t,j} \right) + \mathbb{P} \left( A_{t_0,j} \cap B^c_{t_0,j} \cap C^c_{t,j} \right) \\ 
& \le \mathbb{P} \left( B_{t_0,j}\right) + \mathbb{P} \left( C_{t_0,j}\right) +  \mathbb{P} \left( A_{t_0,j} \cap B^c_{t_0,j} \cap C^c_{t,j} \right).
	\end{split}
\end{equation}
However, notice we have the following inclusion of events
\[
B^c_{t_0,j} \cap C^c_{t,j} \subset \left\lbrace d_{t_0}(j) \le (b_4b_3 + b_5)\sqrt{t_0}\log(t) \right\rbrace.
\]
Thus, choosing $b_8 = 2(b_4b_3 + b_5)$ we have $A_{t_0,j} \cap B^c_{t_0,j} \cap C^c_{t,j} = \emptyset$, which allows us to conclude that
\[
\mathbb{P} \left( A_{t_0,j}\right) \le 2t^{-100}.
\]
Finally, an union bound over $t_0$ followed by a union bound over $j$ implies the desired result.
\end{proof}

\section{Positive local clustering}\label{sec:clusloc}

In this section we prove Theorem \ref{thm:local}, which says that the local clustering coefficient is bounded away from 0 with high probability. 
\begin{proof}[Proof of Theorem \ref{thm:local}] We must find a lower bound for
\[
\mathcal{C}^{loc}_{G_t}  := \frac{1}{t}\sum_{v\in G_t} \mathcal{C}_{G_t}(v).
\]
Let $v_m$ be a vertex in $G_t$ whose degree is $m$. Observe that each \textit{TF-step} we took when $v_m$ was added increase $e\left( \Gamma_{G_t}(v_m) \right)$ by one. So, denote by $T_v$ the number of \textit{TF-steps} taken at the moment of creation of vertex $v$. Since all the choices of steps are made independently, $T_v$ follows a binomial distribution with parameters $m-1$ and $p$. Now, for every vertex we add to the graph, put a blue label on it if $T_v \ge 1$. The probability of labeling a vertex is bounded away from zero and we denote it by $p_b$.

By Theorem \ref{teo:powerlaw}, with probability at least $1-t^{-100}$, we have
\[
N_t(m) \ge b_1t - b_2\sqrt{t\log(t)}.
\]
Thus, the number of vertices in $G_t$ of degree $m$ which were labeled, $N^{(b)}_t(m)$, is bounded from below by a binomial random variable, $B_t$, with parameters $ b_1t - b_2\sqrt{t\log(t)}$ and $p_b$. But, about $B_t$ we have, for all $\delta >0$,
\[
\P\left(B_t \le \frac{\Ed[B_t]}{4}\right) \le \left(1-p_b+p_be^{-\delta}\right)^{b_1t - b_2\sqrt{t\log(t)}}\exp\left\lbrace \delta p_b  \left(b_1t - b_2\sqrt{t\log(t)} \right)/4\right \rbrace
\]
and choosing $\delta$ properly we conclude that, \textit{w.h.p}, 
\begin{equation}\label{eq:ntb}
N_t^{(b)}(m) \ge  p_b\left(b_1t - b_2\sqrt{t\log(t)}\right)/4.
\end{equation}
Finally, note that each blue vertex of degree $m$ has $\mathcal{C}_{G_t}(v) > 2m^{-2}$. Combining this with (\ref{eq:ntb}) we have
\begin{equation*}
\begin{split}
\mathcal{C}_{G_t}^{loc} & > \frac{1}{t}\sum_{v \in N_t^{(b)}(m)} \mathcal{C}_{G_t}(v) > t^{-1}N_t^{(b)}(m)2m^{-2} \\
 &> t^{-1}2m^{-2}p_b\left(b_1t - b_2\sqrt{t\log(t)}\right)/4 \\
 & \to 2m^{-2}p_bb_1 >0 \mbox{ as }t\to+\infty, 
\end{split}
\end{equation*}
proving the theorem.
\end{proof}
\section{Vanishing global clustering}\label{sec:cluglo}
This section is devoted to the proof of Theorem \ref{thm:global} which states that the global clustering of $G_t$ goes to zero at $1/\log t$ speed. Since the proof depends on estimates for the number of cherries, $C_t$, we first derive the necessary bounds and finally put all the pieces together at the end of this section.

\subsection{Preliminary estimates for number of cherries}

Let 
\[\tilde{C}_t:=\sum_{j=1}^td^2_t(j)\]
denote the sum of the squares of degrees in $G_t$. We will to prove bounds for $\tilde{C}_t$ instead of proving them directly for $C_t$. Since $C_t = \tilde{C}_t/2 - mt$ the results obtained for $\tilde{C}_t$ directly extend to $C_t$.
\begin{lema}\label{lema:varctil} There is a positive constant $B_3$ such that
\[
\mathbb{E}\left[ \left( \tilde{C}_{s+1} - \tilde{C}_s\right)^2 \middle| G_s\right] \le B_3\frac{d_{max}(G_s)\tilde{C}_s}{s}.
\]
\end{lema}
\begin{proof}We start the proof noticing that for all vertex $j$ we have~$d^2_{s+1}(j) - d^2_s(j) \le 2md_s(j) + m^2$ deterministically 
From this remark the inequality below follows.
\begin{equation}
\tilde{C}_{s+1} - \tilde{C}_s \le 2m\sum_{j=1}^s d_s(j)\mathbb{1}\left\lbrace \Delta d_s(j) \ge 1 \right\rbrace + 2m^2. 
\end{equation}
Since all vertices have degree at least $m$, we have $m^2 \le m\sum_{j=1}^s d_s(j)\mathbb{1}\left\lbrace \Delta d_s(j) \ge 1 \right\rbrace$, thus
\begin{equation}
\tilde{C}_{s+1} - \tilde{C}_s \le 4m\sum_{j=1}^s d_s(j)\mathbb{1}\left\lbrace \Delta d_s(j) \ge 1 \right\rbrace. 
\end{equation}
Applying Cauchy-Schwarz to the above inequality, we obtain
\begin{equation}
\begin{split}
\left( \tilde{C}_{s+1} - \tilde{C}_s \right)^2 &\le 16m^2\left(\sum_{j=1}^s d_s(j)\mathbb{1}\left\lbrace \Delta d_s(j) \ge 1 \right\rbrace \cdot \mathbb{1}\left\lbrace \Delta d_s(j) \ge 1 \right\rbrace \right)^2 \\
& \le 16m^2\left(\sum_{j=1}^s d^2_s(j)\mathbb{1}\left\lbrace \Delta d_s(j) \ge 1 \right\rbrace\right)\left( \sum_{j=1}^s\mathbb{1}\left\lbrace \Delta d_s(j) \ge 1 \right\rbrace\right)\\
& \le 16m^3\sum_{j=1}^s d^2_s(j)\mathbb{1}\left\lbrace \Delta d_s(j) \ge 1 \right\rbrace.
\end{split}
\end{equation}
Recalling that 
\[
\mathbb{P}\left( \Delta d_s(j) \ge 1 \middle| G_s\right)\le \frac{d_s(j)}{2s}
\]
we have
\[
\mathbb{E}\left( \left( \tilde{C}_{s+1} - \tilde{C}_s \right)^2 \middle| G_s\right) \le B_3\sum_{j=1}^s \frac{d^3_s(j)}{s} \le B_3\frac{d_{max}(G_s)\tilde{C}_s}{s},
\]
concluding the proof.
\end{proof}
\begin{teorema}[Upper bound for $C_t$]\label{teo:cotasupcs}There is a positive constant $B_1$ such that
\[
\mathbb{P}\left( C_t \ge B_1t\log^2(t) \right) \le t^{-98}.
\]
\end{teorema}
\begin{proof} We show the result for $\tilde{C}_t$, which is greater than $C_t$. To do this, we need to determine~$\mathbb{E}[d^2_{t+1}(j)| G_t]$.

As in proof of Proposition \ref{prop:martingaldt}, write 
\[
d_{t+1}(j) = d_t(j) + \sum_{k=1}^m \mathbb{1} \left\lbrace Y_{t+1}^{(k)} = j \right \rbrace
\]
and denote $\sum_{k=1}^m \mathbb{1} \left\lbrace Y_{t+1}^{(k)} = j \right \rbrace$ by $\Delta d_t(j)$. Thus, 
\[
d^2_{t+1}(j) = d^2_t(j)\left(1+ \frac{\Delta d_t(j)}{d_t(j)} \right)^2 = d^2_t(j) + 2d_t(j)\Delta d_t(j) + (\Delta d_t(j))^2 .
\]
Combining the above equation with (\ref{eq:yk}) and (\ref{eq:expecdt}), we get
\begin{equation}\label{eq:d2}
\mathbb{E}\left[ d^2_{t+1}(j) \middle| G_t\right]= d^2_t(j) + \frac{d^2_t(j)}{t}+\mathbb{E}\left[ (\Delta d_t(j))^2 \middle| G_t\right].
\end{equation}
Dividing the above equation by $t+1$, we get
\begin{equation}
\mathbb{E}\left[ \frac{d^2_{t+1}(j)}{t+1} \middle| G_t\right]= \frac{d^2_t(j)}{t}+\frac{\mathbb{E}\left[ (\Delta d_t(j))^2 \middle| G_t\right]}{t+1}, 
\end{equation}
which implies
\begin{equation}\label{eq:ctilde}
\mathbb{E}\left[\frac{\tilde{C}_{t+1}}{t+1} \middle| G_t\right] = \frac{\tilde{C}_t}{t} + \frac{m^2}{t+1} + \sum_{j=1}^t \frac{\mathbb{E}\left[ (\Delta d_t(j))^2 \middle| G_t\right]}{t+1}.
\end{equation}
It is straightforward to see that $\Delta d_t(j) \le (\Delta d_t(j))^2 \le m\Delta d_t(j)$, which implies
\[
\mathbb{E}\left[ (\Delta d_t(j))^2 \middle| G_t\right] = \Theta\left( \frac{d_t(j)}{t}\right).
\]
Thus, (\ref{eq:ctilde}) may be written as
\begin{equation}\label{eq:expeccs}
\mathbb{E}\left[ \frac{\tilde{C}_{t+1}}{t+1} \middle| G_t\right] = \frac{\tilde{C}_t}{t} + \Theta\left(\frac{1}{t}\right).
\end{equation}
Now, define 
\[
X_t := \frac{\tilde{C}_{t+1}}{t+1}.
\] 

Equation (\ref{eq:expeccs}) states that $X_t$ is a martingale up to a term of magnitude $\Theta\left(\frac{1}{t}\right)$. In order to apply martingale concentration inequalities, we decompose $X_t$ as in Doob's Decomposition theorem. $X_t$ can be written as $X_t = M_t + A_t$, in which $M_t$ is a martingale and $A_t$ is a predictable process. By Equation (\ref{eq:expeccs}), we have
\begin{equation}
A_t = \sum_{s=2}^t \mathbb{E}\left[X_s \middle| G_{s-1}\right] - X_{s-1} = \sum_{s=2}^t \Theta\left(\frac{1}{s}\right).  
\end{equation}
I.e., $A_t = \Theta\left(\log(t) \right)$ almost surely.

The remainder of the proof is devoted to controlling $X_t$'s martingale component using Freedman's Inequality. Once again, by Doob's Decomposition Theorem, we have
\[
M_t := X_0 + \sum_{s=2}^t X_s -\mathbb{E}\left[X_s \middle | G_{s-1}\right].
\]
Observe that $M_{t+1} = M_t + X_{t+1} -\mathbb{E}\left[X_{t+1} \middle | G_t\right]$, thus
\begin{equation}
\begin{split}
\left | \Delta M_s \right | &= \left| X_{s+1} -\mathbb{E}\left[X_{s+1} \middle | G_{s}\right]\right| \le \left | X_{s+1} - X_s \right| + \frac{b_9}{s} \\
& \le \left | \frac{\tilde{C}_{s+1} - \left(1+ \frac{1}{s}\right)\tilde{C}_s}{s+1}\right| + \frac{b_9}{s} \\
& \le b_{10} \frac{d_{max}(G_s)}{s} + b_{11}\frac{\tilde{C}_s}{s^2} + \frac{b_9}{s}.
\end{split}
\end{equation}
Since $\Delta\tilde{C_s}$ attains its maximum when the vertices of maximum degree in $G_s$ receive at least a new edge at time $s+1$. Furthermore, since $d_{max}(G_s) \le ms$ and $\tilde{C}_s \le m^2s^2$, there exists a constant $b_{12}$ such that $\max_{s\le t}\left | \Delta M_s \right | \le b_{12}$ almost surely.

Combining
\[
\left | \Delta M_s \right | \le \left | \frac{\tilde{C}_{s+1} - \left(1+ \frac{1}{s}\right)\tilde{C}_s}{s+1}\right| + \frac{b_9}{s}
\]
with Cauchy-Schwarz and Lemma \ref{lema:varctil}, we obtain positive constants $b_{13},b_{14}$ and $b_{15}$ such that
\begin{equation}\label{ineq:expecvarctil}
\begin{split}
\mathbb{E}\left[ (\Delta M_s)^2\middle| G_s\right] & \stackrel{CS}{\le} b_{13} \frac{\mathbb{E}\left[ (\Delta \tilde{C}_s)^2\middle| G_s\right]}{s^2} + b_{14}\frac{\tilde{C}^2_s}{s^4}+ \frac{b_{15}}{s^2} \\
& \stackrel{\textit{Lemma } \ref{lema:varctil}}{\le} b_{16}\frac{d_{max}(G_s)\tilde{C}_s}{s^3}+ b_{14}\frac{\tilde{C}^2_s}{s^4}+ \frac{b_{15}}{s^2}.
\end{split}
\end{equation}

Now, define $V_t$ as 
\[
V_t := \sum_{s=2}^t \mathbb{E}\left[ (\Delta M_s)^2\middle| G_s\right]
\] 
and call \textit{bad set} the event below
\[
B_t := \bigcup_{j = 1}^t \bigcup_{t_0 = j}^t \left\lbrace d_{t_0}(j) > b_8\sqrt{t_0}\log(t) \right\rbrace,
\]
observe that Lemma \ref{lema:cotaunidto} guarantees $\mathbb{P}(B_t) \le 2t^{-98}$. 

Also notice that~$\tilde{C_s} \le b_{17} d_{max}(G_s) s$ almost surely and in $B_t^c$ we have $d_{max}(G_s) \le b_8 \sqrt{s}\log(t)$ for all $s\le t$. Then, outside $B_t$ we have
\begin{equation}\label{ineq:vt}
\begin{split}
V_t & \stackrel{(\ref{ineq:expecvarctil})}{\le} \sum_{s=2}^t b_{16}\frac{d_{max}(G_s)\tilde{C}_s}{s^3}+ b_{14}\frac{\tilde{C}^2_s}{s^4}+ \frac{b_{15}}{s^2} \\
&  \le \sum_{s=2}^t \frac{b_{16}b^2_8b_{17}s^2\log^2(t)}{s^3} + \frac{b_{14}b^2_{17}s^3\log^2(t)}{s^4} + \frac{b_{15}}{s^2} \\
& \le b_{18}\log^3(t).
\end{split}
\end{equation}

So, by Freedman's inequality, we obtain
\[
\mathbb{P}\left( M_t > \l, V_t \le b_{18}\log^3(t) \right) \le \exp \left\lbrace - \frac{\l^2}{2b_{18}\log^3(t) + 2b_{12}\l/3} \right\rbrace.
\]
Therefore, if $\l = b_{19}\log^2(t)$ with $b_{19}$ large enough, we get
\begin{equation}
\mathbb{P}\left( M_t >  b_{19}\log^2(t), V_t \le b_{18}\log^3(t) \right) \le t^{-100}.
\end{equation}
The inequality (\ref{ineq:vt}) guarantees the following inclusion of events
\begin{equation}
B_t^c \subset \left\lbrace V_t \le b_{18}\log^3(t) \right \rbrace.
\end{equation}
And also,
\begin{equation}\label{eq:events}
\left\lbrace X_t \ge b_{21}\log^2(t) \right \rbrace\subset \left\lbrace M_t \ge (b_{21}-b_{20})\log^2(t) \right \rbrace.
\end{equation}
since $A_t\le  b_{20}\log(t)$ and $M_t \ge X_t - b_{20}\log(t)$.

Finally, 
\begin{equation*}
\begin{split}
\mathbb{P}\left( M_t >  b_{19}\log^2(t) \right) &= \mathbb{P}\left( M_t >  b_{19}\log^2(t), V_t \le b_{18}\log^3(t) \right)\\ & \quad + \mathbb{P}\left( M_t >  b_{19}\log^2(t), V_t > b_{18}\log^3(t) \right) \\
& \le t^{-100} + \mathbb{P}\left( B_t \right) \\ 
\mathbb{P}\left( M_t >  b_{19}\log^2(t) \right) & \le 3t^{-98},
\end{split}
\end{equation*}
proving the Theorem..
\end{proof}
We notice that from equation (\ref{eq:ctilde}) we may extract the recurrence below
\[
\mathbb{E}\left[ \tilde{C}_{t} \right] = \left( 1 + \frac{1}{t-1}\right)\mathbb{E}\left[ \tilde{C}_{t-1} \right] + c_0,
\]
in which $c_0$ is a positive constant depending on $m$ and $p$ only. Expanding it, we obtain
\[
\mathbb{E}\left[ \tilde{C}_{t} \right] = \prod_{s=1}^{t-1}\left( 1 + \frac{1}{s}\right)\mathbb{E}\left[ \tilde{C}_{1} \right] + c_0\sum_{s=1}^{t-1}\prod_{r=s}^{t-1}\left( 1 + \frac{1}{r}\right),
\]
which implies $\mathbb{E}[ \tilde{C}_{t}] = \Theta(t\log t)$. This means the upper bound for $\tilde{C}_t$ given by Theorem \ref{teo:cotasupcs} is exactly $\mathbb{E}[ \tilde{C}_{t}]\log(t)$. 

\subsection{The bootstrap argument} Obtaining bounds for $C_t$ requires some control of its quadratic variation, which requires bounds for the maximum degree and $C_t$, as in Lemma \ref{lema:varctil}. Applying some deterministic bounds and upper bounds on the maximum degree we were able to derive an upper bound for $C_t$, which is of order $\mathbb{E}[C_t] \log t$. To improve this bound and obtain the right order, we proceed as in proof of Theorem \ref{teo:cotasupcs}, but making use of the preliminary estimatejust discussed. This is what we call \textit{the bootstrap argument}. 

The result we obtain is enunciated in Theorem \ref{teo:ctweaklaw} and consist of a \textit{Weak Law of Large Numbers}, which states that $C_t$ divided by $t\log t$ actually converges in probability to a constant depending only on $m$. 

\begin{proof}[Proof of Theorem \ref{teo:ctweaklaw}]In proof of Theorem \ref{teo:cotasupcs}, we decomposed the process $X_t = \tilde{C_t}/t$ in two components: $M_t$ and ~$A_t$. The first part of the proof will be dedicated to showing that $M_t = o(\log(t))$, \textit{w.h.p}. Then we show that $A_t = (m^2+m)\log(t)$ also \textit{w.h.p.}

We repeat the proof given for Theorem \ref{teo:cotasupcs}, but this time we change our definition of \textit{bad set} to
\[
B_t = \bigcup_{s=\log^{1/2}(t)}^t \left\lbrace \tilde{C_s} \ge b_{20}s\log^2(s) \right\rbrace.
\]
By Theorem \ref{teo:cotasupcs} and union bound, $\mathbb{P}(B_t) \le \log^{-97/2}(t)$.  Observe that an upper bound for $\tilde{C_s}$ gives an upper bound for $d_{max}(G_s)$, since 
\[
d_{max}^2(G_s) \le \tilde{C_s} \implies d_{max}(G_s) \le \sqrt{\tilde{C_s}} \implies d_{max}(G_s) \le \sqrt{s}\log(s),
\]
when $\tilde{C_s} \le s\log^2(s)$.

Using (\ref{ineq:expecvarctil}) we have, in $B_t^c$,
\begin{equation}
\begin{split}
V_t & \le \sum_{s=1}^{t-1} b_{16}\frac{d_{max}(G_s)\tilde{C_s}}{s^3} + b_{14}\frac{\tilde{C_s}^2}{s^4} + \frac{b_{14}}{s^2} \\
& \le \sum_{s=1}^{\log^{1/2}(t)-1}b'_{16} + \sum_{s=\log^{1/2}(t)}^{t-1} b_{17}\frac{\sqrt{s}\log(s)s\log^2(s)}{s^3} + b_{18}\frac{s^2\log^4(s)}{s^4} + \frac{b_{14}}{s^2} \\
& \le b_{19}\log^{1/2}(t),
\end{split}
\end{equation}
since $d_{max}(G_s) \le m\cdot s$ and $\tilde{C_s} \le 2m^2 \cdot s^2$ for all $s$ and, in $B_t^c$, $d_{max}(G_s) \le \sqrt{b_{20}s}\log(s)$ and $\tilde{C_s} \le b_{20}s\log^2(s)$ for all $s \ge \log^{1/2}(t)$. Then, by Freedman's inequality,
\begin{equation}\label{eq:mt}
\mathbb{P}\left( M_t \ge \log^{1/4+\d}(t), V_t \le b_{19}\log^{1/2}(t)\right) = o(1).
\end{equation}

Recall equation (\ref{eq:ctilde})
\[
\mathbb{E}\left[\frac{\tilde{C}_{t+1}}{t+1} \middle| G_t\right] = \frac{\tilde{C}_t}{t} + \frac{m^2}{t+1} + \sum_{j=1}^t \frac{\mathbb{E}\left[ (\Delta d_t(j))^2 \middle| G_t\right]}{t+1}.
\]
Now, we recall from Lemma \ref{lema:probdeltadt} that for all $k \in \{ 1,...,m\}$
\[
\mathbb{P}\left( Y^{(k)}_{t+1} =v \middle | G_t \right) = \frac{d_t(v)}{2mt}.
\]
Furthermore, 
\[
\mathbb{P}\left( Y^{(k)}_{t+1} =v, Y^{(j)}_{t+1} =v  \middle | G_t \right) = O\left(\frac{d^2_t(v)}{t^2} \right).
\]
Thus,
\begin{equation}
\mathbb{E}\left[ (\Delta d_t(v))^2 \middle| G_t \right] = \frac{d_t(v)}{2t} + O\left(\frac{d^2_t(v)}{t^2} \right),
\end{equation}
which implies that
\[
\mathbb{E}\left[\frac{\tilde{C}_{t+1}}{t+1} \middle| G_t\right] = \frac{\tilde{C}_t}{t} + \frac{m^2+m}{t+1} + O\left(\frac{\tilde{C}_t}{t^3} \right),
\]
and consequently
\[
A_t = \sum_{s=2}^t \frac{m^2+m}{s+1} + O\left(\frac{\tilde{C}_s}{s^3} \right).
\]
As we have already noticed before, all the constants involved in the big-O notation do not depend on the time or the vertex.
From the above equation we deduce that, in $B_t^c$, 
\[
\sum_{s=2}^t O\left(\frac{\tilde{C}_s}{s^3} \right) \le b_{11}\sum_{s=2}^{\log^{1/2}(t)}\frac{\tilde{C}_s}{s^3} + \sum_{s=\log^{1/2}(t)}^{t}\frac{s\log^2(s) }{s^3} \le b_{12} \log(\log(t)).
\]
Thus, in $B_t^c$,
\[
A_t = (m^2+m)\log(t) + o(\log(t)).
\]

Finally, fix a small positive $\e$
\begin{equation}\label{eq:weaklaw}
\begin{split}
\mathbb{P}\left( \left| \frac{\tilde{C}_t}{t\log(t)} - m^2+m \right| > \e \right) & = \mathbb{P}\left( \left| \frac{M_t+A_t}{\log(t)} - m^2+m \right| > \e \right)\\
& \le \mathbb{P}\left( \left| \frac{M_t+A_t}{\log(t)} - m^2+m \right| > \e, B_t^c \right) + \mathbb{P}\left( B_t\right).
\end{split}
\end{equation}
We also have that $B_t^c \subset \{ V_t \le b_{19}\log^{1/2}(t)\}$, which, combined to (\ref{eq:mt}), implies
\[
\mathbb{P}\left( \left| \frac{M_t+A_t}{\log(t)} - m^2+m \right| > \e,M_t \ge \log^{1/4+\d}(t) ,B_t^c \right) = o(1).
\]
And recall that, in $B_t^c$, $M_t$ is at most $\log^{1/4+\d}(t)$ and $A_t = (m^2+m)\log(t)+o(\log(t))$, thus
\[
\mathbb{P}\left( \left| \frac{M_t+A_t}{\log(t)} - m^2+m \right| > \e,M_t < \log^{1/4+\d}(t) ,B_t^c \right) = 0
\]
for large enough $t$.

Recalling that $C_t = \tilde{C}_t/2 - mt $ we obtain the desired result.
\end{proof}

\subsection{Wrapping up}

Until here we devoted our efforts to properly control the number of cherries in $G_t$. Now, we combine these results with simple bounds for the number of triangles in $G_t$ to finally obtain the exact order of the global clustering.
\begin{proof}[Proof of Theorem \ref{thm:global}] By Theorem \ref{teo:ctweaklaw} we have 
\[
C_t  = \Theta\left(t\log(t)\right), \textit{w.h.p}.
\]
But, observe that number of triangles in $G_t$, $\D_{G_t}$, is bounded from above by $\binom{m}{2}t$. And note that every \textit{TF-step} we take increases $\D_{G_t}$ by one. Then,
\[
\D_{G_t} \ge Z_t = \sum_{s=1}^t T_s
\]
where $T_s$ is the number of \textit{TF-steps} we took at time $s$. Since all the choices concerning which kind of step we follow are independent, $T_s \sim bin(m-1,p)$~and~$Z_t \sim bin((m-1)t,p)$. By Chernoff Bounds, $Z_t \ge \delta(m-1)pt$, for a small $\delta$, \textit{w.h.p.} Thus,
\[
\D_{G_t} = \Theta(t), \textit{w.h.p},
\]
which conclude the proof.
\end{proof}

\section{Final comments on clustering}\label{sec:finalcomm}
We end this paper by comparing the two clustering coefficients from a different perspective than in Section \ref{sec:heuristics}. Recall that $\mathcal{C}^{loc}_{G_t }$ is an unweighted average of local clustering coefficients.
\[
\mathcal{C}^{loc}_{G_t } := \frac{1}{t}\sum_{v\in G_t} \mathcal{C}_{G_t}(v) .
\]
On the other hand, $\mathcal{C}_{G_t}^{glo}$ is a weighted average, where the weight of vertex $v$ is the number of cherries that it belongs to,
\begin{equation}\label{eq:newc1}
\mathcal{C}^{glo}_{G_t } = 3\times \frac{\sum_{v\in G_t} \mathcal{C}_{G_t}(v) \binom{d_t(v)}{2}}{\sum_{v\in G_t} \binom{d_t(v)}{2}}
\end{equation}
Thus the weight of $v$ in  $\mathcal{C}^{glo}_{G_t }$ is basically proportional to the square of the degree. This skews the distribution of weights towards high-degree nodes. 
The clustering of the high degree vertices is the reason why the two coefficients present so distinct behavior. 

We will show below that $\mathcal{C}_{G_t}(v)$ for a vertex $v$ of high degree $d$ is of order $d^{-1}$, which explains why $\mathcal{C}^{glo}_{G_t }$ goes to zero. Recall that the \textit{r.v.} $e_t(\Gamma_v)$ counts the number of edges between the neighbors of $v$. Due to the definition of our model, one can only increse $e_t(\Gamma_v)$ by one if  $d_t(v)$ is also increased by at least one unit. Since $e_t(\Gamma_v)$ can only increase by $m$ units in each time step, we have:
\[
e_t(\Gamma_v) \le m\,d_t(v),
\]
which implies an upper bound for $\mathcal{C}_{G_t}(v)\approx e_t(\Gamma_v)/d_t(v)^2$ of order $d_t^{-1}(v)$. The next proposition gives a lower bound of the same order.
\begin{proposicao}Let $v$ be a vertex of $G_t$. Then, there are positive constants, $b_1$ and $b_2$, such that
\[
\P\left(\mathcal{C}_{G_t}(v) \le  \frac{b_1}{d_t(v)}  \middle |  d_t(v) \ge b_2 \log(t) \right) \le t^{-100}.
\]
\end{proposicao}
This proposition does not prove our clustering estimates, but seems interesting in any case. 

\begin{proof} [Proof of the Proposition] Observe that if we choose $v$ and take a \textit{TF-step} thereafter, we increase $e_t(\Gamma_v)$  by one. Then, if we look only at times in which this occurs, $e_t(\Gamma_v)$ must be greater than a binomial random variable with parameters: \textit{number of times we choose $v$ at the first choice} and $p$. Since all the choices concerning the kind of step we take are made independently of the whole process, we just need to prove that \textit{number of times we choose $v$ at the first choice}, denoted by $d^{(1)}_t(v)$, is proportional to $d_t(v)$ \textit{w.h.p}. 

Recall that $Y^{(1)}_s$ indicates the vertex chosen at time $s$ at the first of our $m$ choices. The random variable $d^{(1)}_t(v)$ can be written in terms of $Y^{(1)}$'s just as follows
\begin{equation}
d^{(1)}_t(v) = \sum_{s=v+1}^t \mathbb{1}\left\lbrace Y^{(1)}_s = v \right\rbrace.
\end{equation}
We first claim that if $d_t(v)$ is large enough, a positive fraction of its value must come from~$d_t^{(1)}(v)$.
\begin{claim} There exists positive constants $b_1$ and $b_2$ such that
\[
\P \left(d^{(1)}_t(v) \le b_1 d_t(v) \middle| d_t(v) \ge b_2\log(t) \right) \le t^{-100}. 
\]
\end{claim}
\begin{claimproof}To prove the claim we condition on all possible trajectories of $d_t(v)$. In this direction, let $\omega$ be an event describing when $v$ was chosen and how many times at each step. We have to notice that $\omega$ does not record whether $v$ was chosen by a \textit{PA-step} or a \textit{TF-step}. The event $\omega$ can be regarded as a vector in $\{0,1,...,m \}^{t-v-1}$ such that $\omega(s) = k$ means we chose $v$ $k$-times at time $s$. For each $\omega$, let $d_{\omega}(v)$ be the degree of $v$ obtained by the sequence of choices given by $\omega$. 

Recall the Equations (\ref{eq:pns}) and (\ref{eq:pss}) which states that
\begin{equation*}
\begin{split}
\mathbb{P}_t\left( Y_{t+1}^{(i+1)}=v \middle | Y_{t+1}^{(i)}\neq v \right) & = \frac{d_t(v)}{2mt} + O\left( \frac{d^2_t(v)}{t^2} \right)
\end{split}
\end{equation*}
and
\begin{equation*}
\begin{split}
\mathbb{P}_t\left( Y_{t+1}^{(i+1)}=v \middle | Y_{t+1}^{(i)} = v \right) & = (1-p)\frac{d_t(v)}{2mt}.
\end{split}
\end{equation*}
For any $\omega$ such that $\omega(s) = k \ge 1$, we may show, using (\ref{eq:pns}) and (\ref{eq:pss}), that there exists a positive constant $\delta$ depending only on $m$ and $p$, such that
\begin{equation}
\P\left( Y^{(1)}_s = v \middle| \omega \right) \ge \delta.
\end{equation}
Furthermore, given $\omega$, the random variables $\mathbb{1}\{Y^{(1)}_s = v \}$ are independent. This implies that, given $\omega$, the random variable $d_t^{(1)}(v)$ dominates stochastically another random variable following a binomial distribution of parameters $d_{\omega}(v)/m$ and $\delta$. Thus, by Chernoff bounds, we can choose a small $b_1$ such that
\[
\P \left( d_t^{(1)}(v) \le b_1 d_{\omega}(v) \middle | \omega \right) \le \exp \left( - d_{\omega}(v) \right).
\]
Since we are on the event $D_t := \{d_t(v) \ge b_2\log(t) \}$, all $d_{\omega}(v) \ge b_2\log(t)$ for some $b_2$ that can be chosen in a way such that
\[
\P \left( d_t^{(1)} \le b_1 d_{\omega}(v) \middle | \omega \right) \le t^{-100}
\]
for all $\omega$ compatible with $D_t$.
Finally, to estimate $\left\lbrace d^{(1)}_t(v) \le b_1 d_t(v)\right\rbrace$, we condition on all possible history of choices $\omega$
\begin{equation}
\begin{split}
\P \left(d^{(1)}_t(v) \le b_1 d_t(v) \middle| D_t \right) & = \sum_{\omega} \P \left(d^{(1)}_t(v) \le b_1 d_t(v) \middle|  \omega, D_t \right)\P \left(\omega \middle| D_t \right) \le t^{-100}
\end{split}
\end{equation}
and this proves the claim.
\end{claimproof}

As we observed at the beginning, $e_t(\Gamma_v)$ dominates a random variable $bin(d^{(1)}_t(v), p)$. And by the claim, $d^{(1)}_t(v)$ is proportional to $d_t(v)$, \textit{w.h.p.} Using Chernoff Bounds, we obtain the result.
\end{proof}
\appendix

\section{Power law degree distribution}

\begin{lema}[Lemma 3.1 \cite{ChungLu}]\label{lema:convseq}
Let $a_t$ be a sequence of positive real numbers satisfying the following recurrence relation
\[ a_{t + 1} = \left( 1 - \frac{b_t}{t} \right) a_t + c_t. \]
Furthermore, suppose $b_t \rightarrow b > 0$ and $c_t \rightarrow c$, then
\[
\lim_{t \rightarrow \infty}\frac{a_t}{t} = \frac{c}{1+b}.
\]
\end{lema}
\begin{proof}[Proof of Theorem \ref{teo:powerlaw}] We divide the proof into two parts. Part \textit{(a)} is the power law for the expected value of the proportion of vertices with degree $d$. Part \textit{(b)} is the concentration inequalities $N_t(d)$.

\textit{Proof of part (a).}The is essentially the same gave in Section 3.2 of  \cite{ChungLu}. The key step is obtain a recurrence relation involving $\mathbb{E}[N_t(d)]$ which has the same form of that required by Lemma \ref{lema:convseq}.
To obtain the recurrence relation, observe that $N_{t+1}(d)$ can be written as follow 
\begin{equation}\label{eq:recuntd}
 N_{t+1} ( d) = \sum_{v \in N_t ( d)} \mathbbm{1}_{\{ \Delta d_t (v) = 0
     \}} + \sum_{v \in N_t ( d - 1)} \mathbbm{1}_{\{ \Delta d_t (v) = 1 \}} +
     \ldots + \sum_{v \in N_t ( d - m)} \mathbbm{1}_{\{ \Delta d_t (v) = m \}}.
\end{equation}
Taking the conditional expected value with respect to $G_t$ on the above equation, applying Lemma \ref{lema:probdeltadt}  and recalling that~$N_t(d) \le t$, we obtain
\[ 
\mathbb{E} [ N_{t+1} ( d) | G_t ] = N_t ( d) \left[ 1 - \frac{d}{2t} +O\left(\frac{d^{2}}{t^{2}}\right) \right] + N_t ( d -1) \left[ \frac{( d - 1)}{2t} + O\left(\frac{(d-1)^{2}}{t^{2}}\right) \right] + O \left( \frac{1}{t} \right) .
\]
Finally, taking the expected value on both sides, denoting $\mathbb{E}N_t (d)$ by $a_t^{( d)}$, we have
  \begin{equation}\label{eq:ad}
    a_{t + 1}^{( d)} = \left[ 1 - \frac{\frac{d}{2} + O \left(
    \frac{d^2}{t} \right)}{t} \right] a_t^{( d)} + a_t^{( d - 1)} \left[
    \frac{\frac{d - 1}{2} + O \left( \frac{(d-1)^2}{t} \right)}{t} \right]
    + O \left( \frac{1}{t} \right).
  \end{equation}
From here, the proof follows by an induction on $d \ge m$ and application of Lemma \ref{lema:convseq}, assuming~$\frac{\mathbb{E}N_t (d -1)}{t} \longrightarrow D_{d-1}$, which gives us
  \[ \frac{a_t ( d)}{t} \longrightarrow \frac{D_{d - 1} \frac{( d - 1)}{2}}{1
     + d / 2} = D_{d-1} \frac{d - 1}{2 + d} = : D_d \]
and this gives us that 
 \[ D_d = \frac{2}{2 + m} \prod_{k = m + 1}^d \frac{( k - 1)}{k + 2} = \frac{2(m+3)(m+1)}{(d+3)(d+2)(d+1)}
  \]
  which proves the part \textit{(a)}.

\textit{Proof of part (b).} The proof is in line with proof of Theorem 3.2 in \cite{ChungLu}. For this reason we show only that following process
\[
X_t^{(d)} := \frac{N_t(d) - D_dt + 16d\cdot c\cdot \sqrt{t}}{\psi_d(t)},
\]
in which $\psi_d(t)$ is defined as 
\[
\psi_d(t) := \prod_{s=d}^{t-1}\left( 1 -\frac{d}{2s}\right)
\]
is a submartingale and we give an upper bound for its variation.

As in Theorem 3.2 of \cite{ChungLu}, the proof follows by induction on $d$.\\
\textit{Inductive step}:  Suppose that for all $d'\le d-1$ we have
\begin{equation}\label{ineq:claimlower}
\mathbb{P}\left( N_t(d') \le D_{d'}t - 16d'\cdot c\cdot \sqrt{t} \right) \le (t+1)^{d'-m}e^{-c^2}.
\end{equation}
Recalling that 
\[
\mathbb{E}\left[ N_{t+1}(d) \middle | G_t\right] = \left(1 - \frac{d}{2t} + O\left( \frac{d^2}{t^2} \right) \right)N_t(d)+\frac{(d-1)N_t(d-1)}{2t} + O\left( t^{-1} \right)
\]
we have the following recurrence relation
\begin{equation}\label{ineq:psix}
\begin{split}
\mathbb{E}\left[\psi_d(t+1)X_t^{(d)} \middle | G_t\right] & \ge \left(1 - \frac{d}{2t} \right)N_t(d)+\frac{(d-1)N_t(d-1)}{2t}\\
& +O\left( t^{-1}\right) + 16d\cdot c\cdot \sqrt{t} - D_d(t+1).
\end{split}
\end{equation}
The inductive hypothesis assure us that
\[
N_t(d-1) \ge D_{d-1}t - 16(d-1)\cdot c\cdot \sqrt{t},
\]
with probability at least $1-(t+1)^{d-1-m}e^{-c^2}$. Thus, returning on (\ref{ineq:psix}),
\begin{equation}\label{ineq:psix2}
\begin{split}
\mathbb{E}\left[\psi_d(t+1)X_t^{(d)} \middle | G_t\right] & \ge \left(1 - \frac{d}{2t} \right)N_t(d) \\
& + \frac{(d-1)D_{d-1}}{2} -D_{d}(t+1)+ 16d\cdot c\cdot \sqrt{t} + O\left( t^{-1}\right).
\end{split}
\end{equation}
But, observe that about the \textit{r.h.s} of the above inequality, we have
\begin{equation}
\begin{split}
\frac{(d-1)D_{d-1}}{2} -D_{d}(t+1)+ 16d\cdot c\cdot \sqrt{t} + O\left( t^{-1}\right) & \ge \left(1 - \frac{d}{2t} \right)\left( -D_dt +16d\cdot c\cdot \sqrt{t}\right) \\
& \iff \\
\frac{(d-1)D_{d-1}}{2} -D_d + O\left( t^{-1}\right) & \ge \frac{dD_d}{2t} - \frac{8d^2c}{\sqrt{t}} \\ 
& \iff \\
\frac{(d-1)D_{d-1}}{2}+ \frac{8d^2c}{\sqrt{t}} + O\left( t^{-1}\right) & \ge D_d + \frac{dD_d}{2t} \\
\end{split}
\end{equation}
but the last inequality is true since we have $(d-1)D_{d-1} = (d+2)D_d$ and $D_d = \frac{2(m+3)(m+1)}{(d+3)(d+2)(d+1)}$. Returning to (\ref{ineq:psix2}), we just proved that $X_{t+1}^{(d)}$ is a submartingale with fail probability bounded from above by $(t+1)^{d-m}e^{-c^2}$. And its variation $\Delta X_t^{(d)}$ satisfies the upper bound below
\begin{equation}
\begin{split}
\left| \Delta X^{(d)}_s \right| & \le \frac{\Delta N_s(d)+D_d + 16dcs^{-1/2} + dN_s(d)(2s)^{-1}}{\psi_d(s+1)} \\
& \le \frac{m + 2/(d+2)+17dc s^{-1/2}+d/2}{\psi_d(s+1)} \\
& \le \frac{2d}{\psi_d(s+1)}+ \frac{17dc}{\sqrt{s}\psi_d(s+1)},
\end{split}
\end{equation}
since $\Delta N_s(d) \le m$, $N_s(d) \le s$ and $D_d \le 2/(d+2)$ for all $s$ and $d$. Thus, there is a positive constant $M$, such that
\begin{equation}\label{ineq:deltaxs2}
\left| \Delta X^{(d)}_s \right|^2  \le \frac{16d^2}{\psi_d^2(s+1)}+\frac{Md^2c^2}{s\psi_d^2(s+1)}.
\end{equation}

The lower bound for $N_t(d)$ is proven applying Theorem 2.36 of \cite{ChungLu} on $X_t^{(d)}$, setting
\[
\lambda = 2c \cdot \sqrt{\sum_{s=d}^{t+1}\left| \Delta X^{(d)}_s \right|^2}.
\]
The upper bound is obtained the same way, but considering the process
\[
^{-}X_t^{(d)} := \frac{N_t(d) - D_dt - 16d\cdot c\cdot \sqrt{t}}{\psi_d(t)},
\]
which is a supermartingale.
\end{proof}

\bibliography{ref}
\bibliographystyle{plain}

\end{document}